\documentclass[12pt]{amsart}
\usepackage{amsmath, amsfonts, amssymb, amsthm, array, enumerate,
  hyperref, pdfsync}
\usepackage[foot]{amsaddr}

\usepackage[top=1.25in,bottom=1.25in,left=1.25in,right=1.25in]{geometry}

\newtheorem{thm}{Theorem}[section]
\newtheorem{prop}[thm]{Proposition}
\newtheorem{lemma}[thm]{Lemma}

\theoremstyle{definition}  
\newtheorem{example}[thm]{Example}

\newtheorem{defn}[thm]{Definition}

\theoremstyle{remark}  
\newtheorem{remark}[thm]{Remark}

\newcommand{\calB}{\ensuremath{\mathcal{B}} }
\newcommand{\calC}{\ensuremath{\mathcal{C}} }

\newcommand{\D}{\displaystyle}

\newcommand{\frakf}{\ensuremath{\mathfrak{f}} }
\newcommand{\frakh}{\ensuremath{\mathfrak{h}} }
\newcommand{\fraki}{\ensuremath{\mathfrak{i}} }
\newcommand{\frakm}{\ensuremath{\mathfrak{m}} }
\newcommand{\frakn}{\ensuremath{\mathfrak{n}} }
\newcommand{\frakv}{\ensuremath{\mathfrak{v}} }
\newcommand{\frakz}{\ensuremath{\mathfrak{z}} }
\newcommand{\frakso}{\ensuremath{\mathfrak{so}} }
\newcommand{\fraksu}{\ensuremath{\mathfrak{su}} }

\newcommand{\ad}{\operatorname{ad}}
\newcommand{\Id}{\operatorname{Id}}
\newcommand{\End}{\operatorname{End}}

\newcommand{\rank}{\operatorname{rank}}
\newcommand{\myspan}{\operatorname{span}}
\newcommand{\trace}{\operatorname{trace}}

\newcommand{\Spec}{\operatorname{Spec}}

\newcommand{\boldC}{\ensuremath{\mathbb C}}

\newcommand{\boldN}{\ensuremath{\mathbb N}}

\newcommand{\boldR}{\ensuremath{\mathbb R}}

\newcommand{\la}{\langle}
\newcommand{\ra}{\rangle}

\begin{document}

\title{A Different Perspective on H-like Lie Algebras}

\author{Cathy Kriloff}
\email{krilcath@isu.edu}
\author{Tracy Payne}
\email{payntrac@isu.edu}
\address{Department of Mathematics and Statistics, Idaho State
  University, Pocatello ID 83209-8085}


\begin{abstract}
We characterize
H-like Lie algebras in terms of  subspaces of cones over conjugacy classes 
in $\frakso(\boldR^q)$, translating the classification problem for H-like 
Lie algebras to an equivalent problem in linear algebra.  We study properties 
of H-like Lie algebras, present new methods for constructing them, including 
tensor products and central sums, and classify H-like Lie algebras whose 
associated $J_Z$-maps have rank two for all nonzero $Z$.
\end{abstract}

\maketitle

\section{Introduction}

A {\em nilmanifold} is a simply connected nilpotent Lie group endowed
with a left-invariant Riemannian metric. 
Kaplan defined the class of nilmanifolds of {\em Heisenberg type}
(also called {\em H-type})~\cite{kaplan-80}.  Nilmanifolds in this class have
many extraordinary symmetry properties (\cite{eberlein-2, btv}). Most notably, they
admit solvable extensions, {\em Damek-Ricci spaces}, which are
harmonic spaces but not necessarily symmetric spaces (\cite{damek-ricci2, btv}).

It is natural to consider  defining conditions for nilmanifolds weaker than those in the definition of
H-type, with the hope that some nice properties of H-type nilmanifolds may be
preserved in the larger class.   Many generalizations of the H-type property 
have been studied.  The {\em
nonsingular} nilpotent Lie algebras were first studied by M\'etivier and later
studied by others from both geometric and analytic perspectives under
various names (\cite{metivier-80, eberlein-1, muller-seeger-04,
kaplan-tiraboschi-13, lauret-oscari-14}).  For {\em
pseudo H-type} or {\em generalized H-type} nilmanifolds, the inner
product is assumed to be nondegenerate instead of positive definite,
and even possibly sub-semi-Riemannian (\cite{ciatti-00, godoy-et-al-13}).
Lauret generalized the notion of H-type to {\em modified H-type} 
in~\cite{lauret-99}, and this was further generalized in~\cite{jang-et-al-08}.  

  In analyzing properties of the length spectrum of two-step
nilmanifolds, Gornet and Mast defined the notion of an {\em H-like}
nilmanifold (\cite{gornet-mast-00}) as a generalization of an H-type
nilmanifold.  Some of the remarkable geometric properties possessed by
H-type nilmanifolds have natural analogs held by H-like nilmanifolds
(\cite{gornet-mast-00, decoste-demeyer-mast}).
 In this work we study the metric Lie algebras
associated to 
H-like nilmanifolds and examine their properties.  (We will give a precise
definition of $H$-like in Definition~\ref{defn H-like}.)

Quite a few low-dimensional examples of H-like nilmanifolds,
including continuous families of two-step nilmanifolds with center of dimension greater than one, have been
found (\cite{gornet-mast-00, decoste-demeyer-mast, schroeder-thesis}).  
There is a one-to-one correspondence
between nilmanifolds and metric nilpotent Lie algebras.   Two
nilmanifolds are isometric if and only if the corresponding metric
nilpotent Lie algebras are metrically isomorphic
(\cite{wilson-82}).  
In~\cite{decoste-demeyer-mast}, a general method for constructing H-like
 Lie algebras using representations of $\fraksu(2)$ was given, and 
 in~\cite{decoste-demeyer-mainkar} H-like 
Lie algebras defined by undirected, uncolored graphs were
classified.

 Fix a two-step metric nilpotent Lie algebra
$(\frakn,Q)$. Let $\frakz = [\frakn,\frakn]$ and let $\frakv =
\frakz^\perp$. (Note that some authors use $\frakz$ to mean the center
of $\frakn$ instead of the commutator.)  If $\dim \frakz = p$ and
$\dim \frakv = q$, we say that $\frakn$ is of type $(p,q)$. Define $J:
\frakz \to \End(\frakv)$ by
\begin{equation}\label{def J} \la J(Z) X, Y \ra = \la Z, [X,Y] \ra \qquad \text{for $X,Y \in
    \frakv, Z \in \frakz$.}\end{equation}
The map $J$ is linear and $J_Z : \frakv \to \frakv$ is skew-symmetric
for all $Z$.   Furthermore, $J(\frakz)$ is a $p$-dimensional subspace
of $\frakso(\frakv)$ (see Lemma~\ref{p dim}).

Conversely, given finite-dimensional inner product spaces $\boldR^p$ and
$\boldR^q$ and a nonzero linear map $J : \boldR^p \to \frakso(\boldR^q),$  there is a
two-step metric nilpotent Lie algebra $(\frakn,Q)$ defined by $J.$   The underlying
vector  space is 
$\frakn = \frakv \oplus \frakz = \boldR^q \oplus \boldR^p$ and the
inner product $Q$ on $\frakn$ is the
orthogonal sum of the inner products on $\frakv$ and $\frakz.$  The
Lie bracket is defined by \eqref{def J} and the assumptions that
$[\frakv,\frakv]$ is contained in $\frakz$ and $\frakz$  
is central.      Two such nilpotent Lie algebras defined
by $J: \boldR^p \to \frakso(\boldR^q)$ and $J': \boldR^p \to
\frakso(\boldR^q),$  define metrically isomorphic metric Lie algebras if
and only if there are $A \in O(q)$ and $B \in O(p)$ so that 
$J'(BZ)
= A \circ J(Z) \circ A^T$  for all $Z \in \frakz$ (see \cite{gordon-wilson-97}).

Because the maps $J(Z)$ are 
skew-symmetric, any nonzero eigenvalues of
$J(Z)$ are purely imaginary, and the eigenvalues of $J(Z)$
completely determine $J(Z)$ up to conjugacy by an element of
$GL(\frakv)$. We use a multiset $S$ to record the eigenvalues
of $J(Z)$. (A {\em multiset}  is a set $S$ 
endowed with a multiplicity function from $S$ to $\boldN$.) Conjugacy
classes of nonzero skew-symmetric maps in $\End(\frakv)$ are indexed
by multisets $S$ of eigenvalues that are purely imaginary with the
multiplicities of $bi$ and $-bi$ the same for all positive real
numbers $b$.  We will call such a multiset {\em admissible}.

Gornet and Mast originally defined H-like Lie algebras in terms
of totally geodesic submanifolds.  There are
several equivalent characterizations of H-like Lie algebras
(see~\cite{gornet-mast-00} and~\cite{decoste-demeyer-mast}); we take one of these as part (3) in the following.
\begin{defn}\label{defn H-like} Let $(\frakn,Q)$ be a two-step
nilpotent Lie algebra endowed with inner product $Q$.  Let $S$ be an admissible
multiset.  
  \begin{enumerate}
  \item{The metric Lie algebra $(\frakn,Q)$ is {\em H-type} if for all
$Z \in \frakz$, $J_Z^2 = -\|Z\|^2 \Id_\frakv$, where $\Id_\frakv$ is
the identity map for $\frakv$. }
  \item{The metric Lie algebra $(\frakn,Q)$ has {\em constant
        $J$-spectrum} $S$  if
       for all unit 
$Z \in \frakz$, the map $J_Z$ has constant spectrum $S$. 
}
  \item{The metric Lie algebra $(\frakn,Q)$ is {\em H-like} if
      $(\frakn,Q)$  has constant $J$-spectrum and $\frakn$ has no nontrivial 
abelian factors. }
  \end{enumerate}
When $(\frakn,Q)$ has constant $J$-spectrum, the rank of $J_Z$ is
constant for nonzero $Z;$ this number is called the {\em $J$-rank} of
$(\frakn,Q)$.
\end{defn}
Examples are given in Section~\ref{background}.  It follows from the
definitions that
\[ \text{H-type}  \Rightarrow \text{H-like}
  \Rightarrow \text{constant $J$-spectrum}.\]
  
In Gornet and Mast's definition of H-like they take $\frakz$
to be the center of $\frakn$ rather than the commutator, so that any
abelian factor of $\frakn$ lies in $\frakz$. We take $\frakz$ to be the commutator
so that any abelian factor is contained  in $\frakv$.

A  $p$-dimensional subspace  $W$ of $\frakso(\boldR^q)$ defines a two-step
metric nilpotent Lie algebra of type $(p,q)$. 
\begin{defn}\label{standard}   
Let  $W$  be a  $p$-dimensional subspace of $\frakso(\boldR^q).$
Define the metric nilpotent Lie algebra $(\frakn,Q)$ by taking as the
 underlying vector space $\frakn =
W \oplus \boldR^q,$ and endowing this vector space with the inner product $\la ,
\ra^\ast$
so that $W$ and $\boldR^q$ are orthogonal, the restriction of   $\la ,
\ra^\ast$ to  $W$  is given by the restriction of the Frobenius inner
product to $W$, and  $\la ,
\ra^\ast$ is the standard inner product on $\boldR^q$; i.e., the one
with
 respect to
which the  standard basis is orthonormal. 
 The Lie bracket for $\frakn$ is defined by
assuming that $W$ is central and letting $J: W \to \frakso(\boldR^q)$
be the inclusion map:  for $Z \in W$ and $X, Y \in \boldR^q,$
$\la [X,Y], Z \ra^\ast = \la  Z(X), Y\ra_{\boldR^q}.$ 
 Such a metric nilpotent
Lie algebra is called the  {\em standard two-step metric nilpotent Lie
  algebra defined by $W.$}
\end{defn}
Eberlein showed that every two-step metric nilpotent Lie algebra is metrically isomorphic
to a standard  two-step metric nilpotent Lie algebra of type $(p,q)$
(Proposition 3.1.2, \cite{eberlein-geometry}).    The standard metric
nilpotent Lie algebras defined by two $p$-dimensional
subspaces $W_1$ and $W_2$ of $\frakso(\boldR^q)$ are isomorphic if and
only if there exists $g \in GL_q(\boldR)$ so that $W_2 = g W_1 g^T$
(Proposition 3.1.3, \cite{eberlein-geometry}). 

Eberlein used this point of view to study a variety of problems involving two-step
nilpotent Lie algebras, including moduli spaces, 
 optimal metrics,  metrics with geodesic flow invariant Ricci
 tensors, and Riemannian submersions and lattices
 (\cite{eberlein-geometry}, \cite{eberlein-moduli}, \cite{eberlein-submersions-lattices}). 
Gordon and Kerr used the correspondence between standard
 metric two-step nilpotent Lie algebras and subspaces
of $\frakso(\boldR^q)$ to
characterize  Carnot Einstein solvmanifolds in terms of subspaces of $\frakso(\boldR^q)$   that they
called ``orthogonal uniform subspaces''  (\cite{gordon-kerr-01}).

Let $\frakv$ be a $q$-dimensional vector space, where $q \ge 2,$
let $S$ be an admissible multiset of size $q$, and let $\calC_S$ denote the
conjugacy class of matrices in $\frakso(\frakv)$ with eigenvalues given by
the multiset $S$. Let $\boldR \calC_S = \{ r A \, : \, A \in
\calC_S, r \in \boldR \}$ be the cone over the conjugacy class $\calC_S$.  It is a
subset of $\frakso(\frakv).$  
The main innovation of our work is to take the perspective of
Eberlein, using standard nilpotent Lie algebras, and to re-interpret
the definition of H-like using cones over conjugacy class in
$\frakso(\frakv).$  This makes some of the proofs we present here much
more simple than they would be otherwise.

Our main theorem is that for any admissible multiset $S,$ 
there is a correspondence between H-like metric nilpotent 
Lie algebras of type $(p,q)$ whose spectrum is a multiple of $S$ 
and standard two-step metric nilpotent Lie algebras defined by
$p$-dimensional subspaces of $\boldR \calC_S$ in
$\frakso(\frakv)$.  

\begin{thm}\label{convex thm} 
Let $S$ be an admissible multiset with size $q \ge 2$.
Let $\boldR \calC_S$ be the cone over the conjugacy class of elements
of $\frakso(\boldR^q)$ with spectrum $S$.  
\begin{enumerate}
\item  If $W$ is a $p$-dimensional
subspace of $\boldR \calC_S$, then the standard two-step metric nilpotent Lie
algebra of type $(p,q)$ defined by $W$ has constant $J$-spectrum $S$. 
\item If $(\frakn = \frakv \oplus \frakz, Q)$ is a two-step metric nilpotent 
Lie algebra of type $(p,q)$ which has constant $J$-spectrum $S,$ then 
$J(\frakz)$ is a $p$-dimensional subspace of $\boldR \calC_S$, and $(\frakn,Q)$  
is homothetic to a standard two-step metric nilpotent Lie algebra. 
\end{enumerate}
\end{thm}

Therefore the problem of classifying H-like Lie algebras is equivalent
to the problem of classifying linear subspaces of $\frakso(\boldR^q)$
that are contained in subsets of the form $\boldR \calC_S$.  This kind of
problem-- finding sets of matrices so that nonlinear properties such
as rank or spectrum are preserved under linear combinations-- has been
studied since at least as early as Hurwitz and Radon
(\cite{hurwitz-23, radon-22}).
For example, see~\cite{ellia-menegatti, ilic-landsberg-99, eisenbud-harris-88, westwick-87, beasley-81,
flanders-62} for subspaces of
matrices having fixed or bounded rank, see~\cite{boralevi-mezzetti,
fania-mezzetti, manivel-mezzetti} for subspaces of skew-symmetric matrices, 
and see~\cite{skrzynski-02} for subspaces of cones over conjugacy
classes.  

Vector spaces of matrices of maximal rank two were first classified in~\cite{atkinson-83}.  Early classification results are reviewed and obtained as special cases of a more general theorem in~\cite{eisenbud-harris-88} (see Theorem 1.1).
\begin{thm}\label{rank 2 classification}~\cite{atkinson-83}
  A vector space of matrices of rank $\le 2$ is equivalent to one of the following:
  \begin{enumerate}
  \item a subspace of the space of matrices of the form 
$\D{ \left\{
    \begin{bmatrix}
      0 & 0 & \cdots & 0 & \ast \\ 
      0 & 0 & \cdots & 0 & \ast \\ 
      \vdots & \vdots & \ddots & \vdots & \vdots \\
      0 & 0 & \cdots & 0 & \ast \\ 
      \ast & \ast & \cdots & \ast  & \ast \\ 
    \end{bmatrix}
\right\}, }$  
  \item $\frakso(\boldR^3)$.
  \end{enumerate} 
\end{thm}
Even a general classification of skew-symmetric matrices
of rank $4$ analogous to the classification in Theorem~\ref{rank 2 classification}
remains unknown.

In Definition~\ref{tensor product def}, we define a class of metric
Lie algebras by taking the tensor product of a $J$-map for a metric
nilpotent Lie algebra with another linear map, and we show in
Proposition~\ref{tensor product} that if the original metric nilpotent Lie algebra is H-like and the
linear map is symmetric, then  the resulting Lie algebra is
H-like.   We show in Proposition~\ref{J unitary} that
if $(\frakn,Q)$ is H-like, then the J map is unitary.  In
Propositions~\ref{subspace sum},~\ref{submersion} and~\ref{central sum prop} we give new methods for constructing H-like Lie algebras
using subspaces of cones over conjugacy classes; these correspond to
families of block diagonal matrices, Riemannian submersions with
fibers in the center and central sums.

We classify H-like Lie algebras such that $\rank(J_Z) \le 2$
for all $Z \in \frakn$:
\begin{thm}\label{rank two thm} 
 Suppose that $(\frakn,Q)$ is an H-like metric nilpotent Lie algebra
and $\rank(J_Z) \le 2$ for all $Z \in \frakn$. Then $(\frakn,Q)$ is
homothetically isomorphic to one of the following:
\begin{enumerate}
\item{$(\frakn,Q)  = (\frakf_{3,2},Q)$ as in Example~\ref{free-3-2},}
\item{an almost abelian metric Lie algebra  as in Example~\ref{star algebra}.}
\end{enumerate}
\end{thm}
In the first theorem of~\cite{decoste-demeyer-mainkar}, the authors
classify two-step nilpotent Lie algebras defined by graphs which admit H-like
inner products and obtain the H-like Lie algebras listed in Theorem
\ref{rank two thm}.  All Lie algebras defined by a graph have
$\rank(J_Z) \le 2$ for all $Z \in \frakn$, so the classification
in~\cite{decoste-demeyer-mainkar} follows from Theorem~\ref{rank two
thm}.  In higher dimensions, there are continuous families of
nonhomothetic metric Lie algebras 
with $\rank(J_Z) \le 2$ for
all $Z \in \frakn$, while there are only countably many Lie algebras
defined by graphs. Hence,  the classification in Theorem~\ref{rank two thm} is
more general than that in~\cite{decoste-demeyer-mainkar}.

Theorem~\ref{rank two thm} can be proved using
Theorem~\ref{convex thm} and Theorem~\ref{rank 2 classification}.
Instead we provide a new, self-contained proof which uses the language
and properties of $J$-maps and uses the metric throughout the
argument.

The paper is organized as follows.  In Section~\ref{background} we
review properties of multisets and present some examples of H-like Lie algebras.  We discuss properties of H-like Lie algebras in Section~\ref{properties}. 
In Section~\ref{convex}, we prove Theorem~\ref{convex thm} and use it to 
prove Propositions~\ref{subspace sum}, \ref{submersion} and~\ref{central sum prop}. In Section~\ref{classify
rank two}, we prove  
Theorem~\ref{rank two thm}.

\section{Background and examples}\label{background}

We will denote the spectra of matrices as multisets with elements from
$\boldC$. Recall that a multiset from $\boldC$ is a subset of $\boldC$
with multiplicities.  For example, the set $\{i, -i, 0\}$ endowed with multiplicity function $m:
S \to \boldN$ given by $m(i) = m(-i) =2$, $m(0) =1$, and $m(z) =0$ for
all other $z \in \boldC$ is a multiset. We say that the sum of the multiplicities of
a multiset $S$ is its {\em size} and we denote the size of $S$ by
$|S|$. The sum $S_1 \uplus S_2$ of multisets $S_1$ and $S_2$ is the
multiset determined by summing the multiplicity functions for $S_1$
and $S_2$. For a multiset $S$ and nonzero scalar $k$ we define the
multiset $kS$ so that the multiplicity for $kz$ in $kS$ is always
equal to the multiplicity of $z$ in $S$.

Let $S$ be an admissible multiset of size $q$.  The Frobenius norm on
$\frakso(\boldR^q)$ is given by $\| A \|^2 = \trace(AA^T)$.  All
elements $A$ of $\frakso(\boldR^q)$ with spectrum $S$ have the same
Frobenius norm, so we define the norm of the multiset $S$ by
\begin{equation}\label{norm}  
N(S)  = \| A \| =   \sqrt{ \sum_{a \in S} |a|^2 \cdot  m(a)},   
\end{equation}
where $m(a)$ is the multiplicity of $a.$  
For example, the multiset $S = \{i, -i, 3i, -3i\}$ with 
$m(i) = m(-i) =2$ and $m(3i) =m(-3i) = 1$ has $N(S) = \sqrt{22}$.

\begin{example}\label{h-type} 
 Suppose that $(\frakn = \frakv \oplus \frakz,Q)$ is H-type.  If
$\|Z\| = 1$, then $J_Z$ is nondegenerate with eigenvalues in the
multiset $S$ defined over $\{i, -i\}$, where $i$ and $-i$ have multiplicity
$\frac{1}{2} \dim \frakv.$ Then $N(S) =  \sqrt{\dim \frakv}$.
\end{example}

Deformations of inner products on H-type metric Lie algebras may give metrics
on the same underlying Lie algebra that are H-like but not H-type.
This occurs  in the next example.   
\begin{example}\label{h5}
Let $\frakv \cong \boldR^4$ and $\frakz \cong \boldR$ with $Z$  
a basis vector for $\frakz$ and $\{X_1, Y_1,
X_2, Y_2\}$ a basis for $\frakv$. Let  $a$ and $b$ be nonzero real
numbers.   Define $J: \frakz \to \End(\frakv)$ by
  \[  J(Z) =
    \begin{bmatrix}
     0  & -a & 0 &0 \\
      a & 0& 0& 0\\
     0  &0 &0 & -b \\
      0 &0 & b& 0 
    \end{bmatrix},
\]
and let $(\frakn,Q)$ be the resulting metric nilpotent Lie algebra with
orthonormal basis $\{Z, X_1, Y_1, X_2, Y_2\}$. Clearly $(\frakn,Q)$
has constant $J$-spectrum equal to the multiset $\{ai, -ai, bi, -bi
\}$.  When $|a|=|b|=1$, we get the H-type inner product on the
five-dimensional Heisenberg Lie algebra and when $|a| \ne |b|$ and both are nonzero, we get
an H-like inner product that is not H-type.   If we allow exactly one of the
parameters to be zero, then $\frakn$ is isomorphic to $\frakh_3\oplus \boldR^2$, and $(\frakn,Q)$ has constant $J$-spectrum
but is not H-like.  
\end{example}

Gornet and Mast gave the following families of nonisometric H-like 
 Lie algebras of type $(2,4)$ (\cite{gornet-mast-00}).  
\begin{example}\label{gornet-mast-00} 
Let $(a,b) \in \boldR^2 \setminus \{(0,0)\}$,
 and assume that $(c,d) \in \{ \pm (-b,a), \pm (-a,b) \}$. 
 Define for orthonormal $Z_1, Z_2$ 
in $\boldR^2,$
\[  J(Z_1) =
  \begin{bmatrix}
   0  &a &0 & 0 \\
   -a & 0& 0&0 \\
  0  & 0 &0 & b\\
   0 &0 & -b & 0 \\
  \end{bmatrix}, \quad \text{and} \quad
J(Z_2) = 
\begin{bmatrix}
 0 & 0 & c & 0\\ 
 0 & 0 &0 & d \\ 
 -c &0 & 0& 0 \\ 
  0 &-d & 0 & 0 
\end{bmatrix}
.\]
The spectrum of $J(Z)$ for unit $Z$ in $\boldR^2$ 
is  $\{ ai, -ai, bi, -bi\}$. When $a=\pm b$, the resulting 
metric Lie algebra is H-type.   When $a$ and $b$ are both nonzero and 
$|a| \ne |b|,$
  it is H-like but not H-type.
\end{example}

\begin{remark}
  The previous two examples show that  H-like metric nilpotent Lie
  algebras are not necessarily soliton:   Soliton inner products on a fixed nilpotent Lie algebra are
  unique up to scaling, but there may be two nonhomothetic 
inner products on a fixed Lie
  algebra which are both H-like.
\end{remark}

In the next example we consider the free two-step nilpotent Lie algebra on three generators.
\begin{example}\label{free-3-2}
Let $\frakf_{3,2}$ be the free two-step nilpotent Lie algebra on three 
generators.   It has basis $\{E_1, E_2, E_3\} \cup \{F_1, F_2, F_3\},$
and the Lie bracket is determined by the relations
\[  [E_1, E_2] = F_1, \quad [E_2,E_3] = F_2, \quad \text{and} \quad
  [E_1, E_3] = F_3.\]
Let $Q_0$ be the inner product that makes the basis orthonormal.
Then $\frakv = \myspan \{ E_i \}_{i=1}^3$ and $\frakz = \myspan
 \{ F_i \}_{i=1}^3$.  With respect to the basis $\calB = \{E_1,
E_2, E_3\}$ of $\frakv$, the endomorphism $J_{a_1 F_1 + a_2 F_2 + a_3
F_3}$ is given by
\[    [J_{a_1 F_1 + a_2 F_2 + a_3 F_3} ]_\calB= \begin{bmatrix}  
0 & -a_1 & -a_3 \\
 a_1 & 0 & -a_2  \\
 a_3 & a_2 & 0 \\ 
\end{bmatrix}.\]
The square of this mapping has eigenvalues $-(a_1^2 + a_2^2 + a_3^2)$
and $0$, with multiplicities 2 and 1 respectively.  Hence
$(\frakf_{3,2},Q_0)$ has $J$-rank two and is H-like.  Because the $J_Z$
maps are always singular, it is not H-type.
\end{example}

The following family of metric Lie algebras was shown to be  H-like in
~\cite{decoste-demeyer-mainkar} (Example 6).
\begin{example}\label{star algebra}
Let $(\frakn^m,Q_0)$ be the metric nilpotent Lie algebra with
orthonormal basis $\{E_0, E_1, E_2, \ldots, E_m\} \cup \{F_1, \ldots,
F_m\}$ with Lie bracket determined by relations
\[ [E_0, E_k] = F_k \quad \text{for $k=1, \ldots, m$.}\] Note that
$\myspan ( \{ E_1, E_2, \ldots, E_m\} \cup \{F_1, \ldots, F_m\})$ is a
codimension one abelian ideal and that all bracket relations are
determined by $\ad_{E_0}$. Such algebras are called {\em almost
abelian.}  
If $Z = a_1 F_1 + \cdots + a_k F_k$,  the map $J_Z^2$ has eigenvalues $-(a_1^2 + \cdots + a_m^2)$ and $0$, with
multiplicities $2$ and $m-1 = \dim (\frakv )- 2$ respectively.
Therefore, if $\|Z\| = 1,$ the spectrum of $J_Z$ is $\{i,-i,0\}.$ Hence
$(\frakn^m,Q_0)$ has $J$-rank 2 and is H-like.
\end{example}

\begin{example}
Let $(\frakn = \frakv \oplus \frakz,
Q)$ be a two-step metric nilpotent Lie algebra so that the isometry group of $Q$
acts transitively on the unit sphere in $\frakz.$  

For any $Z \in \frakz$, $J_{\phi(Z)} = \phi \circ J_Z \circ \phi^{-1},$ so $J_Z$ and
$J_{\phi(Z)}$ have the same spectrum.  Hence $(\frakn, Q)$
is H-like.  
\end{example}

\section{Properties of H-like Lie algebras}\label{properties}

The next proposition describes how the property of having constant 
$J$-spectrum behaves under direct sums.  
\begin{prop}\label{sums}
Let $(\frakn_1,Q_1)$ and $(\frakn_2,Q_2)$ be metric nilpotent
 Lie algebras which are abelian or two-step.  
Define the metric Lie algebra $(\frakn_1
\oplus \frakn_2, Q)$ with $Q$ so that $Q|_{\frakn_1} = Q_1,
Q|_{\frakn_2} = Q_2$ and $\frakn_1 \perp \frakn_2$. 
Assume $\frakn_1
\oplus \frakn_2$ is non-abelian.
The following are
equivalent:
\begin{itemize}
\item The direct sum  $(\frakn_1 \oplus \frakn_2, Q)$ has constant $J$-spectrum.
\item One of 
$(\frakn_1,Q_1)$ or $(\frakn_2,Q_2)$ is abelian, and the other 
has constant $J$-spectrum. 
\end{itemize} 
\end{prop}

\begin{proof}
Write $\frakn_1 = \frakv_1 \oplus \frakz_1$ and $\frakn_2 = \frakv_2
\oplus \frakz_2$. 

Suppose that $(\frakn_1 \oplus \frakn_2, Q)$ has constant
$J$-spectrum.  If  neither $\frakn_1$ nor $\frakn_2$ is abelian, then 
$\frakz_1$ and $\frakz_2$ are nontrivial, so there exist nonzero 
$Z_1 \in \frakz_1 = [\frakv_1,\frakv_1]$ and nonzero $Z_2 \in \frakz_2
= [\frakv_2,\frakv_2].$  By the constant spectrum hypothesis, the
ranks of $J(Z_1 + Z_2)$, $J(Z_1)$, and $J(Z_2)$ are the same.  
But because $\frakv_1
\cap \frakv_2 = \{0\}$ in $\frakn_1 \oplus \frakn_2$, the rank of
$J(Z_1 + Z_2)$ is the sum of the ranks of $J(Z_1)$ and $J(Z_2)$. Therefore one of
them has rank zero, a contradiction to $Z_1$ and $Z_2$ being in the commutator.
If $\frakn_2$ is abelian, then clearly $\frakn_1$ must have constant spectrum.  

For the converse,
suppose that $(\frakn_1,Q_1)$ is two-step with constant spectrum $S$ and $\frakn_2$ is abelian.  Then
$(\frakn_1,Q_1)$ has nontrivial commutator and 
$J^\frakn=J^{\frakn_1}\oplus 0^{\frakn_2}$ 
where $0^{\frakn_2}$ is the zero map on $\frakv_2$.  The spectrum of
$J^\frakn$ is $S \uplus T,$ where $T$ is $\{0\}$ with multiplicity $\dim(\frakv_2).$
\end{proof}

A simple way to build new H-like Lie algebras from old is by taking
the tensor product of a symmetric map with the $J$-map for an H-like
Lie algebra.   Before proving this we need to define the
metric nilpotent Lie algebras corresponding to $J$-maps 
that are tensor products.

\begin{defn}
\label{tensor product def}
  Suppose that $(\frakn,Q)$ is a two-step metric nilpotent Lie algebra
with $\frakn = \frakv \oplus \frakz$. Let $J : \frakz
\to \End(\frakv)$ be the $J$-map for $(\frakn,Q)$. Let $S: \boldR^m
\to \boldR^m$, where $m >1$, be a linear map which is symmetric with
respect to the standard inner product $\la \cdot , \cdot \ra_{\text{standard}}$ on $\boldR^m$. 

 Define the inner product $Q_S$ on $\frakz \oplus (\frakv \otimes \boldR^m)$
so that $Q_S|_{\frakz} = Q|_{\frakz},$  $\frakz \perp (\frakv \otimes
\boldR^m),$ and  the restriction of $Q_S$ to $\frakv \otimes \boldR^m$
is
determined by
\[ \la X_1 \otimes Y_1, X_2 \otimes Y_2 \ra = \la  X_1, X_2 \ra_Q
  \cdot \la Y_1, Y_2 \ra_{\text{standard}}, \]
where $X_1,X_2 \in \frakv$ and $Y_1, Y_2 \in \boldR^m.$   
Define the  map
$J^S: \frakz \to \End(\frakv \otimes \boldR^m)$ by $J^S(Z) = J_Z
\otimes S$.
Let $\frakn \otimes S = \frakz \oplus (\frakv \otimes \boldR^m)$.
Make $\frakn \otimes S$ into a  Lie algebra by defining the 
bracket using the $J$ map through
Equation \eqref{def J}. 
\end{defn}

\begin{example}
Let $(\frakh_3,Q)$ be the Heisenberg metric Lie algebra with
orthonormal basis $\{ X,Y,Z \}$ where $[X,Y] = Z$. Take orthonormal
basis $\{E_1,E_2\}$ for $\boldR^2$ and let $S:
\boldR^2 \to \boldR^2$ be the map with $S(E_1) = aE_1$ and $S(E_2)
= bE_2$, where $a$ and $b$ are nonzero.  Then the map $J^S: \frakz
\to \End(\myspan\{X,Y\} \otimes \boldR^2)$ as in Definition \ref{tensor product def}
sends $Z$ to $J_Z \otimes S$. The resulting Lie algebra $\frakn
\otimes S$ as in Definition \ref{tensor product def} is isometrically isomorphic to the one in Example~\ref{h5}.
\end{example}

\begin{prop}\label{tensor product}
Suppose that $(\frakn= \frakv \oplus \frakz,Q)$ is a two-step metric nilpotent
Lie algebra. Let $S: \boldR^m \to \boldR^m$ be a nonsingular linear map which is
symmetric with respect to the standard inner product on $\boldR^m$, where $m>1$.

Then, as a multiset, the spectrum of $J^S(Z)$ is the multiset consisting
of all products $\lambda \mu$, where $\lambda \in \Spec(J_Z)$ and $\mu
\in \Spec(S)$, and the multiplicity of $\gamma \in \Spec(J^S(Z))$ is the sum of products of multiplicities, $m(\lambda_i)m(\mu_i)$ where $\lambda_i \mu_i=\gamma$.
Hence, if $(\frakn,Q)$ is H-like, then $\frakn \otimes S$ is H-like.
\end{prop}

\begin{proof}
 First we show that $[\frakn \otimes S, \frakn \otimes S] = \frakz$.
From the definition of $\frakn \otimes S$, the commutator for $\frakn
\otimes S$ is contained in the domain of $J^S$, that is $\frakz$. Fix
unit $Z \in \frakz \subseteq \frakn$. Then there exist orthogonal $X$ and $Y$
in $\frakv$ so that $[X,Y]=Z$ in $\frakn$. Choose unit $U$ to be an eigenvector of
$S$ with (nonzero) eigenvalue $\lambda$. Then
\begin{align*}
 \la Z, [X \otimes U, Y \otimes U] \ra 
&= \la J_Z^S (X \otimes U ), Y \otimes U  \ra  \\
& = \la J_Z X \otimes S(U) , Y \otimes U \ra \\
&=  \la J_Z X \otimes \lambda U , Y \otimes U \ra \\
 &= \la Z , Z \ra \cdot \la \lambda U , U \ra \\  &= \lambda \|Z \|^2 \ne 0.\end{align*}
 Hence $Z \in   [\frakn \otimes S, \frakn \otimes S]$.  Therefore  
 $\frakz \subseteq [\frakn \otimes S, \frakn \otimes S]$.

Because the eigenvalues of the tensor product of maps is the set of
products of eigenvalues of each, it is immediate that if $(\frakn,Q)$
has constant $J$-spectrum, then $\frakn \otimes S$ has constant
$J$-spectrum.
\end{proof}

\begin{example}
The generalized Heisenberg groups defined by Goze and Haraguchi 
in~\cite{goze-haraguchi-82} arise from tensor products.  Their Lie
algebras, which we will call {\em generalized Heisenberg Lie
algebras}, are of the form $\frakm^{r} \otimes I_p$, where $\frakm^{r}$ is
the Lie algebra of dimension $2r+1$ from Example~\ref{star algebra}
and $I_p$ is the $p \times p$ identity matrix.  Taking $r=1$ yields the
$(2p+1)$-dimensional Heisenberg algebra, and taking $p=1$ yields
$\frakm^{r}$. Because the Lie algebras $\frakm^{r}$ support H-like
inner products, Proposition~\ref{tensor product} implies that the
generalized Heisenberg Lie algebras of Goze and Haraguchi admit H-like inner products.
\end{example}

We show that if $(\frakn,Q)$ is an H-like metric Lie algebra,
then the map $J$ for $(\frakn,Q)$ is unitary with respect to a
rescaled Frobenius norm on the image space.

\begin{prop}\label{J unitary} 
 Suppose that $(\frakn = \frakv \oplus \frakz,Q)$ is an H-like Lie algebra with spectrum $S$. Endow $\frakz$ with the inner
product that is the restriction of the inner product $Q$, and give  $\End(\frakv)$ the
rescaled  Frobenius inner product
\[ \la A, B \ra_{\End(\frakv)} = \textstyle \frac{1}{N(S)^2} \trace A
  B^T,\]
 where $N(S)$ is  as in Equation \eqref{norm},  and the
trace and transpose are taken with respect to the restriction of the
inner product $Q$ to $\frakv$. Then
\begin{equation}\label{unitary}
  \la J_Y, J_Z \ra_{\End(\frakv)}  = \la  Y, Z \ra_Q \end{equation}
for all $Y, Z \in \frakz$. 
\end{prop}

\begin{proof}   Let $Z \in \frakz$ be a unit vector. 
Because $(\frakn,Q)$ is H-like,
$\trace J_Z^2 = - N(S)^2$. This and skew-symmetry of $J_Z$ give
\[
 \la J_Z, J_Z \ra_{\End(\frakv)} =  \textstyle \frac{1}{N(S)^2}  \trace
 J_Z J_Z^T  = \textstyle-\frac{1}{N(S)^2} \trace J_Z^2  = 1 .\]
Therefore $J$ maps the unit sphere in $\frakz$ into the unit sphere in
$\End(\frakv)$ so it is unitary.   
\end{proof}
As a consequence, with respect to the unscaled Frobenius norm $\|
\cdot \|_{\text{Frobenius}},$ if $(\frakn,Q)$
is H-like, and $\|Z\| = \|W\|,$ then $\| J_Z\|_{\text{Frobenius}} = \| J_W\|_{\text{Frobenius}}.$

The next proposition describes how the spectrum of a vector in
a subspace of the cone over a conjugacy class depends on its norm. 
 
\begin{lemma}\label{sphere spectrum}
  Let $S$ be an admissible multiset of size $q$.  Let $\boldR \calC_S$
be the cone over the conjugacy class for $S$ in $\frakso(\boldR^q)$.
Then any $A \in \boldR \calC_S$ has
spectrum $\frac{\|A\|_{\text{Frobenius}}}{N(S)} S$, where $\| A \|_{\text{Frobenius}}^2
= \trace (AA^T).$
\end{lemma}

\begin{proof} 
Clearly the statement holds when $A=0$. Let $A \in \boldR \calC_S$ be
nonzero.  Because $A$ is in $\boldR \calC_S,$ its spectrum is a
multiple of $S$, and we can rescale $A$ by $\lambda $ so that $\lambda A$ has
spectrum $S$. As $A$ is skew-symmetric, we may assume that $\lambda >0$.

Choose $B \in \boldR \calC_S$ with spectrum $S$. Then $\|B\| = N(S)$,
where $N(S)$ is as in Equation~\eqref{norm}.  Because they have the same
spectrum, $\lambda A$ and $B$ are conjugate.  Conjugation preserves
the Frobenius norm, so $\lambda \|A\| = \| B\|$.
But $\lambda \Spec(A) = \Spec(B) = S$. Hence
\[ \Spec(A) = \frac{1}{\lambda} S = \frac{\|A\| }{ \| B\|} S = \frac{\|A\| }{ N(S)} S. \]
\end{proof}
We will use the following simple yet crucial lemma in the proof of the
main theorem.  It lies behind the correspondence between nilpotent Lie
algebras of type $(p,q)$ and $p$-dimensional subspaces of
$\frakso(\boldR^q)$ in \cite{eberlein-moduli}.
\begin{lemma}\label{p dim}
 Suppose that $(\frakn = \frakv \oplus \frakz,Q)$ is a two-step
 metric nilpotent Lie algebra of type $(p,q).$ 
Then $J(\frakz)$ is a $p$-dimensional subspace of $\End(\frakv).$
\end{lemma}

\begin{proof}
Let $\{Z_1, \ldots, Z_p\}$ be a basis of $\frakz.$ We want to show
that 
$J(Z_1), \ldots, J(Z_p)$ are independent.  Suppose that there are real numbers
$a_1, \ldots, a_p$ so that $\sum_{i=1}^p a_i J(Z_i) = 0.$  But then  
$J(\sum_{i=1}^p a_i Z_i) = 0.$  Hence $J(Z) \equiv 0$ for $Z =
\sum_{i=1}^p a_i Z_i.$  But for $Z \in [\frakn, \frakn],$  $J_Z
\equiv 0$ if and only if $Z=0.$  Because $\{Z_1, \ldots, Z_p\}$  is
independent,   $a_i  = 0$ for all $i.$
\end{proof}

\begin{remark}
  It can be shown that if $(\frakn,Q)$ is H-like, then the
  restriction of the Ricci endomorphism to $\frakz$ is a constant
  times the identity (see \cite[Lemma 1]{payne10}).  
\end{remark}

\section{The Main Theorem and some of its consequences}\label{convex}
 To prove Theorem~\ref{convex thm}, we show that the bijection 
between two-step metric nilpotent Lie algebras 
$(\frakn,Q)$ of type $(p,q)$ and $p$-dimensional subspaces of 
$\frakso(\boldR^q)$ endowed with the natural inner product  is a bijection between algebras with spectrum a multiple of $S$
and $p$-dimensional subspaces of the cone $\boldR \calC_S$ in
$\frakso(\boldR^q)$.

\begin{proof}[Proof of Theorem~\ref{convex thm}]
Fix $(p,q)$. Let $S$ be an admissible multiset of size $q \geq 2$.
Let $W$ be a $p$-dimensional subspace of $\boldR \calC_S \subseteq
\frakso(\boldR^q)$ and let $(\frakn,Q)$ be the standard two-step
metric nilpotent Lie algebra defined by $W$ as in Definition
\ref{standard}. 
Choose $Z \in \frakz$ with $\| Z\| = 1.$  Then $J(Z) \in W \subseteq \boldR
\calC_S,$ so its spectrum is a scalar multiple of $S.$

Suppose that $(\frakn = \frakv \oplus \frakz, Q)$ is a two-step metric nilpotent Lie algebra of type
$(p,q)$ with 
 constant $J$-spectrum $S$. By Lemma~\ref{p dim}, $J(\frakz)$  is a
$p$-dimensional subspace of $\frakso(\frakv)$.  By definition of
constant $J$-spectrum, $J(Z)$ has spectrum $S$ for all unit $Z$.
Therefore the image of the unit sphere in $\frakz$ is contained in
$\calC_S$. Because rescaling an element of $\frakso(\frakv)$ by
$\lambda \in \boldR$ rescales the spectrum by $\lambda$, all elements
in $J(\frakz)$ are in the cone $\boldR \calC_S$.
\end{proof}

The next proposition describes how to explicitly construct H-like
Lie algebras using bases with a particular form.  Recall that to form the sum 
$S = S_1 \uplus S_2 \uplus \cdots \uplus S_k$ of multisets 
$S_1, \ldots, S_k$, we add the multiplicity functions of $S_1, \ldots, S_k$.
\begin{prop}\label{subspace sum}
For $i=1, \ldots, k$, let $\frakv_i$ be a vector space of dimension at
least two, and let  $S_i$
be an admissible multiset with size equal to $\dim(\frakv_i)$.

Suppose that $W$ is a $p$-dimensional 
subspace of $\frakso(\oplus_{i=1}^k \frakv_i)$
having a basis $\calB = \{ B_1, \ldots, B_p \}$ consisting of vectors
$B_j = \oplus_{i=1}^k A_i^j, j =1, \ldots, p$, where for each $i,$ 
\begin{enumerate}
\item{ $A_i^j \in \frakso(\frakv_i)$ for each $j$,}
\item{for all $j$ the spectrum of $A_i^j$ is $S_i$,}
\item{$\myspan\{ A_i^j \}_{j=1}^p \subseteq \boldR\calC_{S_i}$,  
    and}\label{spectrum factor}
\item{$\{ A_i^j \, : \, j=1, \ldots, p \}$ is orthogonal with respect
    to the Frobenius inner product. }\label{orthogonal}
\end{enumerate}
 Then 
$W$ is a subspace of $\boldR \calC_S$, where $S = S_1 \uplus S_2
\uplus \cdots \uplus S_k$ and $\boldR \calC_S$ is the cone over the conjugacy class
$\calC_S$  in
$\frakso(\oplus_{i=1}^k \frakv_i).$
\end{prop}

\begin{proof}
 Let $\calB = \{ B_1, \ldots, B_p \}$ be a basis for $W$ as in the
statement of the proposition.  Property~\eqref{orthogonal} forces $\calB$
to be orthogonal, and by Property~\eqref{spectrum factor}, all of the vectors in $\calB$ have spectrum $S = S_1 \uplus S_2
\uplus \cdots \uplus S_k$
and Frobenius norm $N(S)$, where $N$ is the function associated to the multiset
$S$ as in Equation~\eqref{norm}.  
Let $C$ in $ W$ be nonzero.  
We may write $C$ as  $C = \sum_{j=1}^p   c_j   (\oplus_{i=1}^k A_i^j)
     =  \oplus_{i=1}^k
  \left( \sum_{j=1}^p  c_j A_i^j \right).$  Then the norm of $C$ is
  given by 
\[
\|C\|^2 = \sum_{i=1}^k \sum_{j=1}^p \| c_j A_i^j \|^2 
= \sum_{j=1}^p c_j^2 \sum_{i=1}^k   N(S_i)^2 
=N(S)^2\sum_{j=1}^p c_j^2,  \]
which shows that $\sum_{j=1}^p  c_j^2 = \|C\|^2/N(S)^2.$ 

The restriction of $C$ to $\frakv_i$ for fixed $i$ is  $C_i =
\sum_{j=1}^p  c_j A_i^j$.  By the same reasoning as 
above,  $\sqrt{\sum c_j^2} =
\|C_i\|/N(S_i)$ for all $i.$
Now we apply Lemma~\ref{sphere spectrum} to 
 $C_i$ in $\boldR\calC_{S_i}$ and see that its spectrum is 
 $\lambda S_i,$ where 
 \[ \lambda =\frac{\| C_i \|}{N(S_i)}=\sqrt{\sum c_j^2} \] 
is independent of $i.$   Summing over $i,$ we deduce that 
the spectrum of $C$ is 
\[ \lambda S_1 \uplus \cdots \uplus \lambda S_k  = \lambda (S_1 \uplus
  S_2 \uplus \cdots \uplus S_k) = \lambda S.\]

Thus, $W \subseteq \boldR \calC_S$.
\end{proof}

\begin{example}\label{subspace example}
Take $\frakv_1 = \boldR^2$ and $\frakv_2 = \boldR^2$. Let $a$ and $b$
be nonzero.  Take
\[  A_1^1 =
  \begin{bmatrix}
    0 & -a \\ a & 0 
  \end{bmatrix} \in \frakso(\frakv_1),  \quad A_1^2 =
  \begin{bmatrix}
    0 & -b \\ b & 0 
  \end{bmatrix} \in \frakso(\frakv_2), \quad \text{and} \quad B_1 = A_1^1 \oplus
  A_1^2 \subseteq \frakso(\frakv_1 \oplus \frakv_2).
\] 
Let $W$ be the subspace of $\frakso(\frakv_1 \oplus \frakv_2)$ spanned
by $A_1^1 \oplus A_1^2$.  The resulting metric Lie algebra is
isometrically isomorphic to 
$\frakh_5$ endowed with an H-like inner product as in Example~\ref{h5}. 
Here $S_1 = \{ ai, -ai\}, S_2 = \{ bi, -bi\}$, and $S_1 \uplus S_2 =
\{ ai, -ai, bi, -bi \}$. 
\end{example} 

\begin{example}
Let $(\frakf_{3,2},Q_0)$ and $J_{F_1}, J_{F_2}, J_{F_3} \in \frakso(\boldR^3)$ be as in Example~\ref{free-3-2}.  By Proposition~\ref{subspace sum}, any basis selected from the linearly
dependent set
\[   \{  J_{F_k} \oplus J_{F_l}  \, : \, k,l=1,2,3 \} \subseteq
  \frakso(\boldR^3) \oplus \frakso(\boldR^3) \subseteq \frakso(\boldR^6)\]
defines a subspace of the cone over the conjugacy class for
$\{i,i,-i,-i,0,0\}$. An example of this type appears in (4.7) of~\cite{fania-mezzetti}.
\end{example} 

The next proposition shows that Riemannian submersions in a certain class
map Lie algebras with constant spectrum onto Lie algebras with constant spectrum.  
In fact, by~\cite{eberlein-submersions-lattices}, the Riemannian submersion of 
the corresponding nilmanifolds has simply connected, flat, totally geodesic fibers.

\begin{prop}\label{submersion}
 Suppose that $(\frakn_1=\frakv_1 \oplus \frakz_1,Q_1)$ and
$(\frakn_2= \frakv_2 \oplus \frakz_2,Q_2)$ are metric Lie algebras,
and $\phi: \frakn_1 \to \frakn_2$ is a surjective homomorphism with
$\ker \phi \subseteq \frakz_1$ such that $\la X, Y \ra = \la \phi(X),
\phi(Y)\ra$ for all $X, Y \in (\ker \phi)^\perp$. If $(\frakn_1,Q_1)$
has constant spectrum $S$, then $(\frakn_2,Q_2)$ has constant
spectrum $S$.
\end{prop}

\begin{proof}
Let $Z_2$ be an arbitrary unit vector in $\frakz_2$. Since the
restriction $\phi_0$ of $\phi$ to $(\ker \phi)^\perp$ is a bijection, 
it is invertible.  Hence $Z_2$ has a unique
pre-image $Z_1 = \phi_0^{-1}(Z_2)$ in $(\ker \phi)^\perp$ and this
pre-image has length one.  Because $(\frakn_1,Q_1)$ has spectrum $S$,
$J_{Z_1}$ has spectrum $S$.

For any $Z \in (\ker \phi)^\perp$, $J_{\phi_0(Z)} = \phi_0 \circ J_{Z}
\circ \phi_0^{-1}$. Hence $J_{Z_2} = \phi_0 \circ J_{Z_1} \circ
\phi_0^{-1}$. Since $J_{Z_2} $ is conjugate to $J_{Z_1}$, it also has
spectrum $S$.

Thus, all vectors in the unit sphere in $\frakz_2$ have spectrum $S$
and $J(\frakz_2) \subseteq \boldR \calC_S$. 
\end{proof}

The central sum construction glues two groups together along a
subgroup; when we glue two nilpotent Lie algebras together along their
centers we call this a {\em central sum}.  This was called
concatenation by Jablonski in~\cite{jablonski-11}.

\begin{defn}\label{central sum defn} 
 Let $\frakn_1$ and $\frakn_2$ be Lie algebras with centers
$Z(\frakn_1)$ and $Z(\frakn_2)$ respectively.  Let $\phi: Z(\frakn_1)
\to Z(\frakn_2)$ be a bijective linear map.
\begin{enumerate}
\item Define the Lie algebra $\frakn_1 +_\phi \frakn_2$ as $(\frakn_1
\oplus \frakn_2)/ \fraki$ where $\fraki$ is the ideal $\{(W ,- \phi(W)
) \, : \, W \in Z(\frakn_1)\}$. This Lie algebra is called the {\em
central sum of $\frakn_1$ and $\frakn_2$ defined by $\phi$.}
\item{Suppose that $Q_1$ and $Q_2$ are inner products on $\frakn_1$
and $\frakn_2$ respectively, and that $\phi$ is an isometry with
respect to $Q_1$ and $Q_2$. Let $\pi: \frakn_1 \oplus \frakn_2 \to
\frakn_1 +_\phi \frakn_2$ be the natural projection map.  Define the
inner product $Q$ on $\frakn_1 +_\phi \frakn_2$ so that the
projections $\pi|_{\frakn_1}$ and $\pi|_{\frakn_2}$ are isometries.
Then the metric nilpotent Lie algebra $(\frakn_1 +_\phi \frakn_2, Q)$
defined by $J$ is called the {\em (metric) central sum of
$(\frakn_1,Q_1)$ and $(\frakn_2,Q_2)$ defined by $\phi$.}}\label{metric central sum}
\end{enumerate}

\end{defn}
It is not hard to check that the inner product $Q$
in Definition~\ref{central sum defn}~\eqref{metric central sum} is well-defined.

This construction may be iterated.  For
example the $(2k+1)$-dimensional Heisenberg algebra may be viewed as a
$(k-1)$-fold central sum of $3$-dimensional Heisenberg algebras.
In Section 6 of~\cite{decoste-demeyer-mast}, the authors present
families of H-like Lie algebras which  are central sums.  We show that in general, central sums of
Lie algebras with constant spectrum again have constant spectrum.

An elementary way to define a subspace of the 
cone over a  conjugacy class $\boldR C_{S}$ is by
taking the direct sum of subspaces of $\boldR C_{S_1}$ and $\boldR
C_{S_2},$ where $S = S_1 \uplus S_2.$
Taking direct sums of subspaces translates to a natural construction of metric
Lie algebra
called the metric central sum.    
Before we state the proposition we make some simple observations
about the structure of metric central sums.  Write $\frakn =
\frakn_1 +_\phi \frakn_2$. 
Because $\pi$ is a surjective
homomorphism, \[ [\frakn,\frakn] = \pi([\frakn_1 \oplus \frakn_2,
\frakn_1 \oplus \frakn_2])= \pi(\frakz_1 \oplus \frakz_2) =
\pi(\frakz_1).\] 
The restriction of the canonical projection
map $\pi: \frakn_1 \oplus \frakn_2 \to \frakn$ to $\frakv_1 \oplus
\frakv_2$ is an isometry, and the restriction of $\pi$ to $\frakz_1
\oplus \{0\}$ or $\{0\} \oplus \frakz_2$ is an isometry to
$\pi(\frakz_1 \oplus \frakz_2)$. We thus can identify $\frakz =[\frakn,\frakn]$ with
$\frakz_1$ or $\frakz_2$. We let $\frakv =
\frakz^\perp$, so $\frakn = \frakv \oplus \frakz$. 
After using the three isomorphisms to make appropriate identifications 
we can view the $J$ map for
$\frakn$ as a map from $\frakz_1$ to $\frakso(\frakv_1 \oplus
\frakv_2)$.
\begin{prop}\label{central sum prop} 
Suppose that $(\frakn_1 = \frakv_1 \oplus \frakz_1,Q_1)$ and
$(\frakn_2 = \frakv_2 \oplus \frakz_2,Q_2)$ are metric nilpotent Lie
algebras with centers $Z(\frakn_1)$ and $Z(\frakn_2)$ respectively.
Let $\phi: Z(\frakn_1) \to Z(\frakn_2)$ be an isometry.
\begin{enumerate}
\item 
  Denote the
$J$ map for the metric central sum $\frakn = \frakn_1 +_\phi \frakn_2$
by
\[ J^{\frakn}: \frakz \cong \frakz_1 \to \frakso(\frakv)\cong
\frakso(\frakv_1 \oplus \frakv_2)\] For $Z = \pi(Z_1) \in \frakz$,
\begin{equation}\label{JnZ} 
J^{\frakn} (Z) =   J^{\frakn_1}(Z_1)  \oplus  J^{\frakn_2}( \phi (Z_1) ) \subseteq \frakso(\frakv_1)
  \oplus \frakso(\frakv_2) \subseteq  \frakso(\frakv_1 \oplus
  \frakv_2).\end{equation}
In particular, if $(\frakn_1,Q_1)$ and $(\frakn_2,Q_2)$ have
constant spectra $S_1$ and $S_2$ respectively, then $(\frakn_1 +_\phi
\frakn_2, Q)$ has constant spectrum $S_1 \uplus S_2$.
\item{The subspace of $\frakso(\frakv)$ corresponding to 
$(\frakn_1 +_\phi \frakn_2, Q)$  is $J^{\frakn_1}(\frakz_1) \oplus
J^{\frakn_2}(\frakz_2) \subseteq \frakso(\frakv_1) \oplus
\frakso(\frakv_2)  \subseteq \frakso(\frakv).$}\label{(2)}
\end{enumerate}
\end{prop}

\begin{proof}
Suppose $\frakn_1$ is type $(p_1, q_1)$ and $\frakn_2$ is type $(p_2,
q_2)$. Because $(\frakn_1,Q_1)$ and $(\frakn_2,Q_2)$ are H-like,
$Z(\frakn_1) = \frakz_1$ and $Z(\frakn_2) = \frakz_2$. Because $\phi$
is an isometry, $p_1 = p_2$.

Let $S_1$ and $S_2$ be the spectra for $(\frakn_1,Q_1)$ and
$(\frakn_2,Q_2)$ respectively.  From Theorem~\ref{convex thm},
$J(\frakz_1)$ is a $p_1$-dimensional subspace of $\boldR\mathcal{C}_{S_1}$
contained in the cone over the conjugacy class $\calC_{S_1}$, and
$J(\frakz_2)$ is a $p_2$-dimensional subspace of $\boldR\mathcal{C}_{S_2}$
contained in the cone over the conjugacy class $\calC_{S_2}$.

We write the map $ J^{\frakn}$ for $(\frakn_1 +_\phi \frakn_2, Q)$ in
terms of $J^{\frakn_1}: \frakz_1 \to \frakso(\frakv_1)$ and
$J^{\frakn_2}: \frakz_2 \to \frakso(\frakv_2)$. Let $X_1, Y_1 \in
\frakv_1, X_2, Y_2 \in \frakv_2$, and $Z_1 \in \frakz_1$. Denote their
images under $\pi$ with bars.  Then
\begin{align*}
  \la J^{\frakn}_{\overline{Z_1}}  \overline{(X_1 + X_2)},
  \overline{(Y_1 + Y_2)} \ra &=
\la  \overline{Z_1} , \overline{[X_1,Y_1]}  \ra + \la  \overline{\phi(Z_1)} ,
                               \overline{[X_2,Y_2]} \ra  \\
&= 
\la  Z_1 , [X_1,Y_1] \ra + \la  \phi(Z_1),
                               [X_2,Y_2]\ra  \\
&= 
\la  J_{Z_1}^{\frakn_1} X_1,Y_1 \ra  + \la  J_{\phi(Z_1)}^{\frakn_2}
                               X_2,Y_2 \ra .
\end{align*}
Thus, Equation~\eqref{JnZ} holds.

Let $Z = \pi(Z_1)$ be a unit vector in $\frakz$. Because $\pi$ maps
$\frakz_1$ onto $\frakz$ isometrically, $Z_1$ is a unit vector in
$\frakz_1$. Hence $J(Z_1)$ has spectrum $S_1$. Because $\phi$ is an
isometry, $\phi(Z_1)$ has norm one.  Therefore, $J(\phi(Z_1))$ has
spectrum $S_2$. It follows that $J^{\frakn}(Z)$ as in Equation~\eqref{JnZ} 
has spectrum $ S_1 \uplus S_2$. By Theorem~\ref{convex
thm}, $(\frakn_1 +_\phi \frakn_2, Q)$ is H-like with constant spectrum
$ S_1 \uplus S_2$.

Statement \eqref{(2)} follows from Equation \eqref{JnZ}.
\end{proof}

Proposition~\ref{J unitary} could have been used to
complete the first part of the proof, along with the fact  that $\{
J(Z_i) \}_{i=1}^{p_1}$ and $\{ J( \phi(Z_i)) \}_{i=1}^{p_1}$ are
orthogonal bases for $\frakso(\frakv_1)$ and $\frakso(\frakv_2)$
respectively.

The tensor product construction of $\frakn \otimes S$ 
 in Definition 
\ref{tensor product def} is a central sum, with 
each summand corresponding to an element of
a linearly independent set of eigenvectors for the symmetric map
$S.$  

\section{Classification of H-like Lie
  algebras with  $J$-Rank two}\label{classify rank two}

Throughout this section we let $E(Z)$ denote the eigenspace for the
nonzero eigenvalue of $J_Z^2$, where $J$ is associated to a two-step
metric nilpotent Lie algebra of $J$-rank two.  Note that if $(\frakn,Q)$ is a two-step
metric nilpotent Lie algebra, $J_Z \not \equiv 0$ for all nonzero $Z
\in \frakz$. Note that when $(\frakn = \frakv \oplus \frakz, Q)$ has
no abelian factors (as in the H-like case), $\oplus_{i=1}^p E(Z_i) = \frakv$, where $\{Z_i\}$ is a basis for $\frakz$.

\begin{lemma}\label{E(Z)}
Let $(\frakn = \frakv \oplus \frakz,Q)$ be a metric nilpotent Lie
algebra. If $(\frakn,Q)$ is H-like and has $J$-rank two, then for any
independent vectors $Z_1$ and $Z_2$ in $\frakz$, 
\begin{enumerate}
\item{$E(Z_1) \ne  E(Z_2)$}\label{neq}
\item $E(Z_1) \cap E(Z_2)$ is one-dimensional, and
\item{if  $Z_1$ and $Z_2$ are orthogonal, and nonzero $X$ is in 
$E(Z_1)\cap E(Z_2)$,  the vectors  $J_{Z_1} X$ and $ J_{Z_2} X$ are nonzero and orthogonal.  }\label{og}
\end{enumerate}
\end{lemma}

Part~(\ref{og}) of the lemma is a special case of Theorem~3.8b 
of~\cite{decoste-demeyer-mast}, which in turn relies on Lemma~3.3
of~\cite{gornet-mast-00}.  We give an
alternate proof.  

\begin{proof} 
Assume that $(\frakn,Q)$ is H-like and $Z_1$ and $Z_2$ are independent vectors in $\frakz$. Then
$Z_1, Z_2 \ne 0$.
If $E(Z_1) = E(Z_2)$, then because $(\frakn,Q)$ is H-like, $\lambda_1
J_{Z_1} = \lambda_2 J_{Z_2}$ for nonzero $\lambda_1, \lambda_2 \in \boldR$.
It follows that $Z = \lambda_1 Z_1 - \lambda_2 Z_2$ has $J_{Z} \equiv 0$.  Hence 
$Z=0$, contradicting the independence of $Z_1, Z_2$.

Now assume $(\frakn,Q)$ has $J$-rank two.  We claim $E(Z_1) \cap E(Z_2) \ne \{0\}$. 
Otherwise $E(Z_1) + E(Z_2)$ is four-dimensional and $J_{Z_1 + Z_2}$ will have rank four, a
contradiction.  Also, $E(Z_1) \cap E(Z_2)$ is not two-dimensional,
because that would lead to the contradiction $E(Z_1) = E(Z_2)$.

Suppose that $Z_1$ and $Z_2$ above are orthogonal and $X \in E(Z_1)\cap E(Z_2)$ is
nonzero.  Then certainly $J_{Z_1} X$ and $ J_{Z_2} X$ are nonzero.
Since $J_{Z_1}$ has rank two, there exists an orthogonal
basis $\{X_1, X_2,\dots,X_p\}$ for $\frakv$ so that $X=X_1$, and with respect to this
basis, $J_{Z_1} = X_1 \wedge X_2 $ in $\frakso(\frakv)$. 
Proposition~\ref{J unitary} forces $J_{Z_1}$ and $J_{Z_2}$ to be 
orthogonal in $\frakso(\frakv)$. Hence, 
 \[ J_{Z_2} = \sum_{  \substack{i < j \\  (i,j) \ne
    (1,2)}}    a_{ij} X_i \wedge X_j .\]
Evaluating $J_{Z_2}$ at $X$ gives $J_{Z_2} X = J_{Z_2} X_1=
\sum_{j=3}^q a_{1j} X_j$. Therefore $J_{Z_1} X$ and $J_{Z_2} X$ are
orthogonal.
\end{proof}

We have seen in Examples~\ref{free-3-2} and~\ref{star algebra}
that the free two-step nilpotent Lie algebra on three generators and
certain two-step nilpotent Lie algebras with codimension one abelian
ideals admit H-like inner products.  In fact, these are all  the 
H-like Lie algebras with $J$-rank two.

\begin{thm}\label{classify J-rank 2}
Suppose that  $(\frakn, Q)$ is H-like and has $J$-rank two.
\begin{itemize}
\item  If $\cap_{Z \ne 0} E(Z)$ is nontrivial, then
it is one-dimensional, and $(\frakn,Q)$ is homothetic to
$(\frakn^p,Q_0)$, the metric Lie algebra with $p$-dimensional center
as in Example~\ref{star algebra}.
\item If $\cap_{Z \ne 0} E(Z) = \{0\}$, then
$(\frakn,Q)$ is homothetic to $(\frakf_{3,2},Q_0)$, the free two-step
nilpotent Lie algebra endowed with the inner product $Q_0$ as in
Example~\ref{free-3-2}.
\end{itemize}
\end{thm}

\begin{proof}
We are classifying up to homothety, so we assume the nonzero elements
of the spectrum are $i$ and $-i$.
We break our argument into cases depending on the value of $p=\dim \frakz$.

First suppose that $p=1$ and $Z$ is a unit vector in $\frakz$. If
$J_Z$ has any zero eigenvalues, then $\frakn$ has an abelian factor, a
contradiction.  Thus, for unit $Z$, $J_Z$ must be conjugate to a
multiple of $X_1 \wedge X_2$. Hence $\frakn$ is isomorphic to
$\frakh_3$. But $\frakh_3$ is isomorphic to $\frakn^1$, and after
rescaling the metric, $(\frakn,Q)$ is isomorphic to $(\frakn^p,Q_0)$
with $p=1$.

Now suppose that $p \ge 2$ and that $\cap_{Z \ne 0} E(Z)$ is one-dimensional.  
Let $\{Z_1, \ldots, Z_p\}$ be an orthonormal basis for $\frakz$. 
Let $X_1$ be a unit vector spanning $\cap_{Z \ne 0} E(Z)$. 
Because of our assumption
that nonzero eigenvalues are $\pm i$, $\{ X_1, J_{Z_j} X_1\}$ is an
orthonormal basis for $E(Z_j)$ for all $j$. By Lemma~\ref{E(Z)} part~(\ref{og}), 
the set $\{J_{Z_1} X_1, \ldots, J_{Z_p} X_1 \}$ is independent
and orthogonal.  Hence $\{X_1, J_{Z_1} X_1, \ldots, J_{Z_p} X_1 \}$ is
an orthonormal basis for $\frakv = \oplus_{i=1}^p E(Z_i)$. It is easy
to check that $(\frakn,Q)$ is isometrically isomorphic to
$(\frakn^p,Q_0)$.

Next suppose that $p \ge 2$ and $\dim \cap_{Z \ne 0} E(Z) \ne 1$. 
Part~(\ref{neq}) of Lemma~\ref{E(Z)} implies $\cap_{Z \ne 0} E(Z) \subseteq
E(Z_1) \cap E(Z_2)$ is at most one-dimensional and hence $ \cap_{Z \ne
0} E(Z) = \{0\}$.  If $p=2$ then
$E(Z_1) \cap E(Z_2)=\{0\}$ would contradict $(\frakn, Q)$ having $J$-rank two. 
Hence we assume from now on that $p \ge 3$.

Let $p=3$ and $\{Z_1, Z_2, Z_3\}$ be an
orthonormal basis for $\frakz$. Again
by Lemma~\ref{E(Z)} we know that $E(Z_1) \cap E(Z_2),
E(Z_2) \cap E(Z_3)$ and $E(Z_3) \cap E(Z_1)$ are one-dimensional.  Choose unit vectors $X_1, X_2, X_3$ with
\[ X_1 \in E(Z_1) \cap E(Z_2), \, X_2 \in E(Z_2) \cap E(Z_3), \,
\text{and} \, X_3 \in E(Z_3) \cap E(Z_1).\] 
The vectors $X_1$, $X_2$, and $X_3$ are linearly independent because $J_{Z_1}$, $J_{Z_2}$, and $J_{Z_3}$
are independent. Thus we have
\[ [X_1, X_2] = \pm Z_2, \, [X_2,X_3] = \pm Z_3, \,  [X_3,X_1]= \pm Z_1, \]
which describes $\frakf_{3,2}$.

Now suppose $p>3$. Let $\{Z_1, Z_2, Z_3, \ldots, Z_p\}$ be an
orthonormal basis for $\frakz$ and choose independent unit vectors $X_1, X_2, X_3$ with
\[ X_1 \in E(Z_1) \cap E(Z_2), \, X_2 \in E(Z_2) \cap E(Z_3), \,
\text{and} \, X_3 \in E(Z_3) \cap E(Z_1)\]
as when $p=3$. 
Consider $E(Z_4)$. By Lemma~\ref{E(Z)}, 
$E(Z_4)$ intersects $E(Z_1)$ in a one-dimensional subspace
spanned by some unit $Y_1 \in \frakv$.  Since the map $\boldR Z \mapsto E(Z)$
from the set of lines in $\myspan \{Z_1,Z_2,Z_3\}$ to the set of
two-planes in $\myspan \{X_1,X_2,X_3\}$ is surjective and $Y_1 \in \myspan
\{X_1,X_3\} $, 
there exists unit $Z$ in $\myspan \{Z_1,Z_2,Z_3\}$ so that $E(Z) \perp Y_1$. 
Note that $Z$ and $Z_4$ are independent.

 Now consider the subspace $F = E(Z_4) \cap E(Z).$    By
 Lemma~\ref{E(Z)}, it is one-dimensional.  
Because $J_{Z_1}, J_{Z_2}, J_{Z_3}$ and $J_{Z_4}$
are independent, the subspace $E(Z_4) \cap  \myspan \{X_1,X_2,X_3\}$
is one-dimensional.
Hence it is spanned by the common element $Y_1.$  But then 
\[  F =  E(Z_4) \cap E(Z) \subseteq E(Z_4) \cap  \myspan
  \{X_1,X_2,X_3\} = \myspan
  \{Y_1\} .\] Since $F$ is one-dimensional, it must be 
spanned by  $Y_1.$    But $Y_1$ is orthogonal to $E(Z),$ a contradiction.
\end{proof}

\bibliographystyle{alpha}
\bibliography{h-like}

\end{document}